\documentclass[10pt]{amsart}

\usepackage{amssymb}
\usepackage{amsmath}
\usepackage{amsthm}
\usepackage{amssymb}

\theoremstyle{plain}
\newtheorem{theorem}[subsection]{Theorem}

\newtheorem*{theorem*}{Theorem}

\newtheorem*{proposition*}{Proposition}

\newtheorem*{lemma*}{Lemma}

\theoremstyle{definition}

\theoremstyle{remark}
\newtheorem*{remark}{Remark}

\newtheorem*{acknowledgements}{Acknowledgements}

\newcommand{\Hom}{\operatorname{Hom}}

\newcommand{\GL}{\operatorname{GL}}

\newcommand{\PGL}{\operatorname{PGL}}
\newcommand{\PSL}{\operatorname{PSL}}

\newcommand{\PSO}{\operatorname{PSO}}

\newcommand{\meas}{\operatorname{\meas}}

\newcommand{\nc}{\newcommand}
\nc{\cal}{\mathcal} 
\nc{\la}{\langle} \nc{\ra}{\rangle}
 \nc{\CA}{\cal A}
 \nc{\CBB}{\cal B}
\nc{\CDD}{\cal D}
\nc{\CE}{\cal E}
\nc{\CF}{\cal F} \nc{\CG}{\cal
G} \nc{\CH}{\cal H} \nc{\CI}{\cal I} \nc{\CJ}{\cal J}
\nc{\CK}{\cal K} \nc{\CL}{\cal L} \nc{\CM}{\cal M} \nc{\CN}{\cal
N} \nc{\CO}{\cal O} \nc{\CP}{\cal P} \nc{\CQ}{\cal Q}
\nc{\CR}{\cal R} \nc{\CS}{\cal S} \nc{\CT}{\cal T} \nc{\CU}{\cal
U} \nc{\CV}{\cal V} \nc{\CW}{\cal W} \nc{\CZ}{\cal Z}

\nc{\fa}{\mathfrak a} \nc{\fg}{\mathfrak g} \nc{\fk}{\mathfrak k}
\nc{\fh}{\mathfrak h} \nc{\fm}{\mathfrak m} \nc{\fn}{\mathfrak n}
\nc{\fA}{\mathfrak A} \nc{\fC}{\mathfrak C} \nc{\fI}{\mathfrak I}
\nc{\fL}{\mathfrak L} \nc{\fS}{\mathfrak S}
\nc{\fz}{\mathfrak z}

\nc{\nen}{\newenvironment} \nc{\ol}{\overline}
\nc{\ul}{\underline} \nc{\lra}{\longrightarrow}
\nc{\lla}{\longleftarrow} \nc{\Lra}{\Longrightarrow}
\nc{\Lla}{\Longleftarrow} \nc{\Llra}{\Longleftrightarrow}
\nc{\hra}{\hookrightarrow} \nc{\iso}{\overset{\sim}{\lra}}

\makeatletter
\@addtoreset{equation}{section}
\makeatother
\numberwithin{equation}{section}
 \nc{\ba}{\mathbb A}
 \nc{\bq}{\mathbb Q}
 \nc{\br}{\mathbb R}
 \nc{\bz}{\mathbb Z}
 \nc{\bc}{\mathbb C}
 \nc{\bn}{\mathbb N}
 \nc{\ck}{\mathcal{K}}
 \nc{\G}{\Gamma}
 \nc{\sm}{\setminus}
 \nc{\sub}{\subset}
 \nc{\lm}{\lambda}
 \nc{\al}{\alpha}
 \nc{\bt}{\beta}
 \nc{\om}{\omega}
 \nc{\dl}{\delta}
 \nc{\g}{\gamma}
 \nc{\Dl}{\Delta}
 \nc{\Om}{\Omega}
 \nc{\s}{\sigma}
 \nc{\ro}{\rho}
 \nc{\te}{\theta}
 \nc{\SLR}{\operatorname{SL}_2(\br)}
 \nc{\GLR}{\operatorname{GL}_2(\br)}
 \nc{\PGLR}{\operatorname{PGL}_2(\br)}
 \nc{\PSLR}{\operatorname{PSL}_2(\br)}
 \nc{\PSLZ}{\operatorname{PSL}_2(\bz)}
 \nc{\SLC}{\operatorname{SL}_2(\bc)}
 \nc{\uH}{\mathbb H}
 \nc{\fD}{\mathcal{D}}
 \nc{\fE}{\mathcal{E}}
 \nc{\fO}{\mathcal{O}}
 \nc{\haf}{\frac{1}{2}}
 \nc{\qtr}{\frac{1}{4}}
 \nc{\shaf}{{\scriptstyle\frac{1}{2}}}
 \nc{\hlm}{{\scriptstyle\frac{\lambda}{2}}}

 \nc{\8}{\infty}
 \nc{\7}{{-\infty}}
 \nc{\inv}{^{-1}}
 \nc{\eps}{\varepsilon}
 \nc{\aG}{\mathbf{G}}
 \nc{\spn}{\operatorname{Span}}
 \nc{\Cm}{\operatorname{CM}}

\begin{document}
\setcounter{tocdepth}{1}
\title[Dirichlet Series]{Torus periods of automorphic functions and the meromorphic continuation of related Dirichlet Series}

\author{Andre Reznikov} \email{}

\address{Department of Mathematics\\
Bar-Ilan University \\
Ramat Gan, 52900 \\
Israel}

\begin{abstract} We consider modular functions (i.e., the Eisenstein series and Hecke-Maass forms) for the group $\PSL_2(\bz)$. We fix  a quadratic number field $E$. This gives rise to  twisted (by a Hecke character of the field $E$)  periods of a modular function along the torus corresponding to $E$.  We prove meromorphic continuation for a  Dirichlet series generated by these twisted periods.
\end{abstract}


\thanks{Partially supported by the Veblen Fund at IAS, by a BSF grant, by the ISF  Center of Excellency grant 1691/10, and by the Minerva  Center at ENI}
\maketitle


\section{Introduction}\label{secintro}

\subsection{} In this paper, we consider $L$-functions and Hecke characters of  quadratic fields. Although arguments presented here are valid for a general  quadratic number field $E$, in order to simplify the presentation, we will first deal with the simplest case of Gauss numbers, and describe the general case in Section \ref{CM}. Hence, let $E=\mathbb{Q}(i)$ be the Gauss field. For an integer $n\in \bz$, consider the  Hecke character of $E$ given by $\chi_n(a)=\big({a}/{|a|}\big)^{4n}$. The set of all such characters $\{\chi_n\}_{n\in\bz}$ could be described as the set $\cal{X}_{un}(E)$ of all (maximally) unramified Hecke characters of $E$. The corresponding Hecke $L$-function  is given by the series $L(s,\chi)=\sum\limits_{a\in I^*(\CO_E)}\chi(a)N(a)^{-s}$, for $Re(s)>1$ (where the summation is over all non-zero integer ideals of $E$, i.e., over $\bz[i]/\{\pm 1,\pm i\}$ for $E=\mathbb{Q}(i)$).

We consider a double Dirichlet series given by
\begin{equation*}
D_{E}(s,w)=L(s,\chi_0)+\sum_{n\in\bz\setminus 0}L(s,\chi_n)|n|^{-w}\ ,
\end{equation*} for $(s,w)\in\bc^2$, $Re(s)\gg 1,\ Re(w)\gg1$. In what follows,  one can omit the first term $L(s,\chi_0)$ from the sum, or take the sum over positive $n$ only.

\begin{theorem}\label{thm-eisen}
The series $D_{E}(s,w)$ defines the function which extends to a meromorphic function on $\bc^2$.
\end{theorem}
It turns out that it is more convenient to consider the  function
\begin{equation}\label{D-tilde-Eisen}\tilde D_E(s,w)=\frac{2^{w/2}}{\G\left(\frac{1-w}{2}\right)}\cdot D_E(s,w)\ , \end{equation}
and we prove the meromorphic continuation for this function.
\subsection{Torus periods} One quickly recognizes that the above theorem is related to periods of automorphic functions. In fact, our proof of the  meromorphic continuation is based on two well-known facts. First, we invoke classical results of E. Hecke \cite{He} (and also of H. Maass and  C. Siegel \cite{Si}) about torus periods of Eisenstein series. Namely, we  consider the automorphic representation $E_s$ corresponding to the (normalized) Eisenstein series $E(s,x)$ for $\PGL_2$ over $\bq$. Let  $T_E\subset\PGL_2$ be the torus corresponding to $E$. The (Fourier) expansion along the orbit $T_E(\bq)\setminus T_E(\ba_\bq)\subset \PGL_2(\bq)\setminus \PGL_2(\ba_\bq)$ is given in terms of Hecke characters of $E$, and naturally leads to Hecke $L$-functions $L(s,\chi)$ over $E$. This allows us to realize the series $D_{E}(s,w)$ as the spectral expansion of the value at (the class of) the identity element $\bar e\in T_E(\bq)\setminus T_E(\ba_\bq)$ for a special vector $v_w$ in the space of the Eisenstein series representation $E_s$. We note that the vector $v_w$ is not a smooth vector, and belongs to an appropriate Sobolev space completion of $E_s$. In particular, we invoke the meromorphic continuation of smooth Eisenstein series as opposed to $K$-finite Eisenstein series (for general a treatment, see \cite{BK}, \cite{L}; for a congruence subgroup of $\PGL_2(\bz)$, an elementary treatment based on Fourier expansion of $E(s)$ is also available).

To prove the meromorphic continuation of the value at the identity for the special vector $v_w$,  we use Hecke operators and the classical technique going back to at least M. Riesz \cite{R} (and might be attributed to Euler) of the analytic continuation strip by strip (which the author learned from the seminal paper \cite{B}). The main observation that allows  us to apply this technique is the fact  that  modulo higher Sobolev spaces, the vector $v_w$  is an approximate eigenvector of (appropriately understood) Hecke operators (see Lemma \ref{aprox-eigenval}).

\subsection{Cusp forms} One can apply the same argument to a Hecke-Maass cusp form $\phi$ instead of the Eisenstein series $E(s)$.   The resulting series is a usual Dirichlet series in {\it one variable} built from the coefficients $a_n$ which are (twisted) torus periods of the cusp form $\phi$. We now define these coefficients.

Let $\phi$ be a Hecke-Maass form for the group $\G=\PGL_2(\bz)$ (one can easily extend our arguments to a congruence subgroup). In particular, $\phi$ is an eigenfunction of the Laplace-Beltrami operator $\Dl$ on the Riemann surface $\PSL_2(\bz)\setminus\CH$ with the eigenvalue which we denote by  $\Lambda(\phi)=(1-\tau^2)/4$ for $\tau=\tau(\phi)\in i\br\cup (0,1)$ (of course, for $\PGL_2(\bz)$, the parameter $\tau$ is pure imaginary, and this is expected to hold for congruence subgroups). We normalize $\phi$ by its $L^2$-norm. We denote by $(V_\tau, \pi_\tau)$ the isomorphism class of the (smooth) automorphic representation of $G=\PGLR$ generated by $\phi$. The structure of such a representation  is well known, and in particular $V_\tau$ has an orthonormal basis of $K$-types $\{e_n\}_{n\in 2\bz}$ which we fix  (here $K=\PSO(2,\br)\simeq S^1$ is a maximal connected compact subgroup of $G$). We consider the Taylor-like expansion of $\phi$ at $z=i$ (generally one considers  a $\Cm$-point $\fz\in \CH$; in fact such an expansion exists at any point of $\CH$). Denote by  $F_\tau(n,g)=\langle \pi_\tau(g) e_0, e_n\rangle_{\pi_\tau}$, $n\in 2\bz$,  matrix coefficients in the representation $\pi_\tau$. Functions $F_\tau(n,g)$ are  right $K$-invariant and hence could be viewed as functions of $z\in\CH\simeq \PGL^+_2(\br)/K$. Functions  $F_\tau(n,z)$  are eigenfunctions of $\Dl$ on $\CH$ with the same eigenvalue as $\phi$, and spherically equivariant $F_\tau\left(n,\left(
                                          \begin{smallmatrix}
\cos\te   &-\sin\te\\
                             \sin\te               &\ \  \cos\te \\
\end{smallmatrix}
                                        \right)z\right)= e^{in\te}F_\tau(n,z)$.  We have the following well-known expansion (first considered by  H. Petersson \cite{P} for holomorphic forms and also by A. Good \cite{Go} in general):
\begin{equation}\label{spherical-exp}
\phi(z)=\sum_{n\in 2\bz}a_n F_\tau(n,z)\ ,
\end{equation} where $a_n=a_n(\phi)\in\bc$. Of course, coefficients $a_n$ depend on the normalization of functions $F_\tau(n,z)$, which we fix in Section \ref{F-normalize} (by choosing a basis $\{e_n\}$ of $\pi_\tau$). This normalization will essentially coincide with one of the classical normalizations of the special function $F_\tau(n,z)$. In particular, this will not depend on $\phi$ but only on the parameter $\tau$.   We note that the analogous expansion is valid  for the Eisenstein series $E(s,z)$ as well, and gives $a_n(E(s))= L(s,\chi_{-n})$.

We consider the Dirichlet series
\begin{equation}\label{D-series-cusp}
D_{E}(\phi,w)=a_0+\sum\limits_{n\in 2\bz\setminus 0} a_n |n|^{-w}\end{equation}
defined for $|w|\gg 1$. (The coefficient $a_0$ is included for the aesthetic reasons only.) As with the Eisenstein series we consider the function
\begin{equation}\label{D-tilde-cusp}\tilde D_E(\phi,w)=\frac{2^{w/2}}{\G\left(\frac{1-w}{2}\right)}\cdot D_E(\phi,w)\ . \end{equation}
\begin{theorem}\label{thm-cusp}

The Dirichlet series $\tilde D_{E}(\phi,w)$  extends to a holomorphic function on $\bc$.
\end{theorem}

\subsection{Hyperbolic periods} A similar treatment is available for real quadratic fields. These correspond to (compact) closed geodesics on the Riemann surface $Y=\PGL_2(\bz)\setminus\PGL_2(\br)$. In fact, from the adelic point of view, there is no difference in the treatment of $\Cm$-points and of closed geodesics.

Let $l\subset Y$ be a closed geodesic. Such a geodesic corresponds to a closed orbit of the diagonal subgroup $A=\left\{\left(
                                          \begin{smallmatrix}
a  & \\
                                   &\ \ b \\
\end{smallmatrix}
                                        \right)\right\}\subset G$ acting on the right on $X=\G\setminus G$. We denote this orbit by the same letter $l\subset X$. Consider the corresponding subgroup $A_l=Stab_A(l)$. We will assume that it is cyclic and we will choose a generator $a_l=diag(a_l,\pm a_l\inv)$. This choice gives the corresponding hyperbolic element $\g_l\in \G$ which is conjugate to $a_l=g_l\g_l g_l\inv$. Eigenvalues of $a_l$ (and of $\g_l$) generate the group of  units in a quadratic field $E$. In fact, there is a finite number of closed geodesics corresponding to the same field and this is reflected in the class number of the field.

For a closed geodesic $l$ as above, we obtain an expansion of automorphic functions similar to the expansion at a $\Cm$-point we discussed above.   Such an expansion is valid for cusp forms and for  Eisenstein series (e.g., see \cite{Gol}). To describe these in classical terms, one introduces special functions similar to $F_\tau(n,z)$ above. This time we use characters $\chi:A_l\setminus A\to\bc^\times$ of the compact group $A_l\setminus A\simeq S^1$. For any such a character, we consider the function $G_\tau(\chi,z)$ on $\CH$ which is an eigenfunction of $\Dl$, is right $K$-invariant, and satisfies $G_\tau(\chi, g_l\inv a g_l\circ z)= \chi(a)G_\tau(\chi,z)$. In fact, we can view the function $G_\tau(\chi,z)$ as defined on the hyperbolic cylinder $\CH_l=\G_l\setminus \CH$ for which $l$ is the ``neck" (i.e., the shortest geodesic). We have to normalize functions $G_\tau(\chi,z)$ which we do in Section \ref{F-normalize} by presenting an explicit integral representation for these functions (the function $G_\tau(\chi,z)$ classically is given in terms of the Gauss hypergeometric function).
We obtain the expansion analogous to \eqref{spherical-exp}
\begin{equation}\label{hyp-exp}
\phi(z)=\sum_{j\in \bz}b_j G_\tau(j,z)\ ,
\end{equation} where $b_j=b_{\chi_j}(\phi)\in\bc$ are the coefficients of this hyperbolic expansion, and functions $G_\tau(j,z)=G_\tau(\chi_j,z)$ are indexed by characters $\chi_j$ in the group  $\widehat{A_l\setminus A}\simeq\bz$. With this we define  the Dirichlet series
\begin{equation}\label{D-series-cusp-hyp}
D_{E}(\phi,w)=b_0+\sum\limits_{j\in \bz\setminus 0} b_j |j|^{-w}\end{equation}
defined for $|w|\gg 1$. (The coefficient $b_0$ is again included out of aesthetic reasons only.) As before, we consider the function
\begin{equation}\label{D-tilde-cusp-hyp}\tilde D_E(\phi,w)=\frac{2^{w/2}}{\G\left(\frac{1-w}{2}\right)}\cdot D_E(\phi,w)\ . \end{equation}

\begin{theorem}\label{thm-cusp}

The Dirichlet series $\tilde D_{E}(\phi,w)$  extends to a holomorphic function on $\bc$.
\end{theorem}

We hope that by denoting the series by the same symbol as in $\Cm$-case, we will not cause too much of a confusion (and in fact from the adelic point of view the treatment of these two cases is identical).

\subsection{Remarks} 1. Under the normalization we chose in \eqref{spherical-exp} and in \eqref{hyp-exp}, coefficients $a_n$ satisfy a mean-vale bound $c_\phi\leq T\inv\sum_{|n|\leq T}|a_n|^2\leq C_\phi$ for appropriate constants $C_\phi,\ c_\phi>0$ as $T\to\8$. Hence there are no obvious reasons why the series $D_E(\phi,w)$ possess the meromorphic continuation. \\
2. Coefficients  $a_n$ and $b_n$ are related to $L$-functions in a  more subtle way than for the Eisenstein series. Namely,  the  theorem of J-L. Waldspurger (see \cite{W}, \cite{JC}, \cite{KW}) provides the relation $|a_n|^2=L(1/2, BC_E(\phi)\otimes \chi_n)/L(1,Ad(\phi))$, where $BC_E(\phi)$ is the base change of the cusp form $\phi$. In spite of this relation,  our method naturally  treats coefficients $a_n$ and not  quantities  $|a_n|^2$.\\
3. The proof that we give shows that  the polar divisor of $D_E(s,w)$ is contained in the union  of the line   $s=1$ with the union of two families of   lines
\begin{eqnarray}\label{divisor}\ w=2-2s-j\ , {\rm or}\ w=1-j, \ j=0,1,2,\dots\  .
\end{eqnarray}
A somewhat more symmetric form of $\tilde D_E(s,w)$ and of its polar set is discussed in Remark \ref{torus-char}.
For a Hecke-Maass cusp form $\phi$, we show that the series $\tilde D_E(\phi,w)$ is holomorphic.
One can also obtain polynomial bounds in $s$ and $w$ for the resulting function  $\tilde D_E(s,w)$.

After announcing the results of the paper, the author was informed by the referee that one can deduce the meromorphic continuation, and the exact locations of poles for the series $D_E(s,w)$ from properties of the Lerch zeta function  (see \cite{Lerch}), and  from the functional equation for the Hecke $L$-functions $L(s,\chi)$.  While this approach is more elementary, it could not cover the case of cusp forms. On the other hand, for the Eisenstein series, this approach leads to the exact location of poles, while our method only gives the potential polar divisor of  $D_E(s,w)$.

It is also apparently possible to deduce meromorphic continuation of $D_E(\phi,w)$ for cusp forms which are not necessary Hecke forms,   using methods of \cite{BLZ}, as was demonstrated by R. Bruggeman (personal communication). It is however less clear how to extend this approach to cover real quadratic fields or the Eisenstein series.
\\
4. One can consider a slightly more general series by prescribing the ramification of Hecke characters. Let $S$ be a finite set of primes, including all  primes ramified in $E$, and $\cal{X}_{S}$  be the set of Hecke characters unramified outside $S$. The natural extension of our method then gives the meromorphic continuation for the series  $\sum_{\chi\in \cal{X}_{S}} L_S(s,\chi)R_S(\chi^S)|\chi_\8|^{-w}$, where $L_S(s,\chi)$ is the partial $L$-function,  $R_S(\chi^S)$ is a rational in $q_i$ function for $q_i$ that are norms of primes in $S$, and $\chi^S$ is the ramified part of $\chi$.\\
5. An important issue in the theory of (double) Dirichlet series is the presence of functional equation(s). The theory of Eisenstein series provides the functional equation in $s$ relating $D_E(s,w)$ and $D_E(1-s,w)$. It is not clear if there is a functional equation involving $w$.\\
6. Finally, we note that from the point of view of the method we present, there is nothing special about series  $D_E(s,w)$ and  $D_E(\phi,w)$. Namely, one can change the weight $|n|^{-w}$ to many other similar weight functions, and still obtain  the meromorphic continuation by the same method. As a result, it is possible that one might have to modify these series in order to study their possible arithmetical properties (e.g., special values).


\section{Torus periods}\label{periods}

We refer to \cite{Bu} for standard facts about automorphic functions and automorphic representations (of real and adele groups).

\subsection{Torus periods of Eisenstein series}\label{tor-periods-set-up} We recall  the classical result of E. Hecke.
We present it in (a more transparent to us) adelic language. Let $G=\PGL_2$.  By specifying an isomorphism $E^\times\subset Aut_\bq(E)\simeq Aut_\bq(\bq^2)$, we obtain   the corresponding tori $T_E\subset \PGL_2$ defined over $\bq$. For $s\in\bc$, $s\not=1$, let $\cal{E}_s\simeq\hat\otimes_{p\leq\8}E_{s,p}$ be the automorphic representation of $G_{\ba_\bq}$ corresponding to the classical {\it normalized} Eisenstein series (given by
$E_s(z)=\sum_{c,d}y^s/|cz+d|^{2s}$ for $Re(s)>1$, where the summation is over $(c,d)\in\bz^2\setminus (0,0)$). In this normalization, the unitary Eisenstein series corresponds to $Re(s)=1/2$. We consider vectors in $\cal{E}$ which are pure tensors of the form $v_\8\otimes v_f\in \cal{E}_s$ where $v_f=\otimes_{p<\8} v_p$ is the standard $K_f$-fixed vector for the maximal compact subgroup $K_f$ of $G$ over finite adeles,  and
$v_\8$ is an arbitrary vector in the infinite component $E_{s,\8}$ of $\cal{E}_s$.  Recall that the theory of Eisenstein series provides the automorphic realization $E_s(v,g)$ (i.e., a function on $X_\ba$) for a vector $v\in\cal{D}_s$ in the principal series representation $\cal{D}_s$ of $G_{\ba_\bq}$, where $\cal{D}_s$ is the space  of homogeneous functions with respect to the $\ba_\bq^\times$ action  on the space $Z_{\ba_\bq}=N_{\ba_\bq}\setminus G_{\ba_\bq}=\prod'_{p\leq\8} N_{\bq_p}\setminus G_{\bq_p}$. The space $\cal{D}_s$ has the natural structure of the (restricted) tensor product $\cal{D}_s\simeq\hat\otimes_{p\leq\8}D_{s,p}$ coming from the above product structure of $Z_{\ba_\bq}$ (unlike the space $\cal{E}_s$ where the tensor product is not described in terms of the underlying space $X_\ba=G_\bq\sm G_\ba$). Here local components are the spaces $D_{s,p}$ of homogeneous functions on $N_p\setminus G_p$. Hence when talking about models of the local representations $E_{s,p}$, we can use the spaces $D_{s,p}$ (in fact, we only use $p=\8$ since we will not discuss  ramified Hecke characters).

\begin{proposition*}\label{Hecke-Maass-prop} (Hecke) Let $v=v_\8\otimes v_f\in \cal{E}_s$ be a vector with everywhere unramified standard finite components. We have then
\begin{equation}\label{Hecke-Maass}
\int_{T_E(\bq)\setminus T_E(\ba_\bq)}E_s(v,t)\chi(t)dt=I_\8(s,\chi_\8,v_\8)\cdot L(s,\chi)\  ,
\end{equation} for any character $\chi=\chi_\8\otimes\chi_f: Z_{G}(\ba_\bq)T_E(\bq)\setminus T_E(\ba_\bq)\to \bc^\times$ (i.e., for a Hecke character of $E$ trivial on $\ba_\bq^\times\subset\ba_E^\times$). Here the functional $I_\8(s,\chi_\8,\cdot):E_{s,\8}\to\bc$ is given by
\begin{equation}\label{I8-funct}
I_\8(s,\chi_\8,v_\8)=\int_{ T_E(\br)}v_\8(t)\chi_\8(t)dt\  ,
\end{equation} for a vector $v\in D_{s,\8}$.
\end{proposition*}
\begin{proof} This follows from, the standard by now, unfolding (see \cite{G}, \cite{Gol}). In fact, for a general $E$ and a Hecke character $\chi$ of $E$, one has to take into account the ramification of $\chi$. In that case, the right hand side of \eqref{Hecke-Maass} takes the form $I_\8(s,\chi_\8,v_\8)\prod_{\nu\in S}I_\nu(s,\chi_\nu,v_\nu)\cdot L_S(s,\chi)\ $ where $S$ is the set of ramified for $\chi$ primes, and $L_S(s,\chi)$ is the partial $L$-function.
\end{proof}

Note that for $E=\bq(i)$, the group $T_E(\br)$ could be naturally identified with the subgroup $ K_\8= \operatorname{PSO}(2,\br)\subset \PGL_2(\br)$ (i.e.,  $K_\8= \operatorname{PSO}(2,\br)$ is the standard maximal connected compact subgroup of $G_\8=\PGL_2(\br)$). Hence for a Hecke character $\chi_n$, the resulting functional $I_\8(s,\chi_{n,\8},\cdot)$ could be identified with the projection to the particular norm one $n$'th $K_\8$-type in the representation $E_{s,p}$ . Moreover, in the realization of $E_{s,\8}$ as $D_{s,\8}$, this functional is given by the integration against the character {\it itself} on the image of the compact subgroup $K_\8\subset Z_\8=N_\8\setminus G_\8$ coming from the archimedean part $(T_E)_\8\subset G_\8$ of the torus $T_E$. For other $\Cm$-fields, we obtain a compact subgroup conjugated to $\operatorname{PSO}(2,\br)$, and hence have to consider types with respect to the corresponding  subgroup.

 By abuse of notations, we denote by $e\in G_{\bq}\setminus G_{\ba_\bq}$  the image of the identity. From the Plancherel formula for ${T_E(\bq)\setminus T_E(\ba_\bq)}$ (i.e., the Fourier expansion w.r.t. characters of  $T_E(\ba_\bq)$ trivial on $T_E(\bq)$), and Proposition \ref{Hecke-Maass-prop} above, we see that the following expansion holds for a vector $v=v_\8\otimes v_f\in \cal{E}_s$:
\begin{equation}\label{E-L-expansion}
E_s(v,e)=\sum_{\chi_n\in \cal{X}_{un}} L(s,\chi_n)\cdot I_\8(s,\chi_{n,\8},v_\8)\  .
\end{equation}

\subsection{Periods of cusp forms} Periods of cusp forms could be defined in the same way as for the Eisenstein series. However there is an important difference concerning their normalization.

Let $\pi\simeq\hat\otimes_{p\leq\8}\pi_p$ be an automorphic  cuspidal representation of $G(\ba)$ in the space of smooth vectors $V_\pi\simeq\hat\otimes_{p\leq\8}V_{\pi_p}$, together with the isometric realization $\nu_\pi: V_\pi\to C^\8(X_\ba)$. Let $T_E(\bq)\setminus T_E(\ba_\bq)\subset \PGL_2(\bq)\setminus \PGL_2(\ba_\bq)$ be the orbit of $T_E(\ba)$. For a Hecke character $\chi$ of $E$, we consider the corresponding $\chi$-equivariant functional $d_\chi^{aut}\in \Hom_{T_E(\ba)}(\pi,\chi)$ given by the integral (as for the Eisenstein series)
\begin{equation}\label{aut-funct-cusp}
d_\chi^{aut}(v)=\int_{T_E(\bq)\setminus T_E(\ba_\bq)}\phi_v(t)\bar\chi(t)dt\  ,
\end{equation}
where $\phi_v=\nu_\pi(v)\in C^\8(X)$ is the automorphic function corresponding to the smooth vector $v\in V_\pi$ under the isometry $\nu_\pi$ (i.e., the automorphic realization of the vector $v$). It is well-know that the space of local equivariant functionals is at most one dimensional $\dim\Hom_{T_p}(\pi_p,\chi_p)\leq 1$, and hence we have a decomposition $d_\chi^{aut}=\otimes_p d_{\chi_p}$ for some local functionals $d_{\chi_p}\in \Hom_{T_p}(\pi_p,\chi_p)$. However, unlike in the case of Eisenstein series, the lack of unfolding for the integral
\eqref{aut-funct-cusp} does not allow us to  choose easily  a specific element in the space $\Hom_{T_p}(\pi_p,\chi_p)$. In fact, one  can normalize $d_{\chi_p}$ up to a constant with absolute value one (this is connected to the Waldspurger theorem alluded before, and discussed in great generality in \cite{IchIke}), but for our purposes it is not enough since we are interested  in the period itself and not in its absolute value. Hence, we will choose the trivial normalization of local functionals in the following way.

We assume that the representation $\pi$ is unramified everywhere, i.e., that the Hecke-Maass form $\phi$ is invariant under the full group $\PSLZ$ (in fact, we can easily deal with any congruence subgroup). The notion of restricted tensor product assumes that we have chosen a $K_p$-invariant vector $e_p\in V_{\pi_p}$ of norm one for every finite $p$ (in general for almost all $p$). We have $\dim\Hom_{T_p}(\pi_p,\chi_p)=1$ since we assumed that $\pi_p$ is unramified for all $p$. It is known (see \cite{GrPr}) that a non-zero invariant functional does not vanish  on the vector $e_p$. We denote by $d_{\chi_p}^{mod}\in \Hom_{T_p}(\pi_p,\chi_p)$ the functional satisfying $d_{\chi_p}^{mod}(e_p)=1$. We consider the corresponding functional
$d_{\chi_f}^{mod}=\otimes_{p<\8}d_{\chi_p}^{mod}$ for finite adeles which clearly satisfies $d_{\chi_f}^{mod}(e_f)=1$ for  $e_f=\otimes_{p<\8}e_p$.  Hence for any choice of a non-zero functional $ d_{\chi_\8}^{mod}\in \Hom_{T_\8}(\pi_\8,\chi_\8)$, we obtain the coefficient of proportionality $a_\chi=a_\chi(\nu_\pi, d_{\chi_\8})\in \bc$ such that

\begin{equation}\label{a-chi-coeff}
d_\chi^{aut}=a_\chi\cdot d_{\chi_\8}^{mod}\otimes d_{f}^{mod} \ .
\end{equation} In fact, since we will only consider vectors of the form $v=v_\8\otimes e_f$, we can write
$d_\chi^{aut}(v_\8)=a_\chi\cdot d_{\chi_\8}^{mod}(v_\8)$.  We now specify the functional $d_{\chi_\8}^{mod}$. As we mentioned, for $E=\bq(i)$, we can naturally identify $T_\8=T_E(\br)$ with the subgroup $K_\8$. Characters of $T_\8$ related  to equivariant functionals on irreducible representations of $G_\8$ are parameterized by even integers $n\in 2\bz$, and naturally correspond to projectors onto (one-dimensional) $K_\8$-types. Let $\{e_n\}_{n\in 2\bz}$ be an orthonormal basis of $V_{\pi_\8}$ consisting of $K_\8$-types. We denote by $d_n^{mod}\in \Hom_{T_\8}(\pi_\8,\chi_n)$ the functional given by $d_n^{mod}(v_\8)=\langle v, e_n\rangle_{\pi_\8}$, where the character $\chi_n$ is given by
$\chi_n\left(\left(
                                          \begin{smallmatrix}
\cos\te   &-\sin\te\\
                             \sin\te               &\ \  \cos\te \\
\end{smallmatrix}
                                        \right)\right)= e^{in\te}$.

Hence we have the decomposition analogous to \eqref{E-L-expansion}
\begin{equation}\label{Cusp-expansion}
\phi_{v_\8}(e):=\nu_\pi(v_\8)(e)=\sum_{\chi_n\in \cal{X}_{un}}a_n\cdot d_n^{mod}(v_\8)\  .
\end{equation} It is easy to see that this expansion coincides with the expansion \eqref{spherical-exp} given in classical terms on $\CH$.

\subsection{Test vectors} In order to realize  the series $D_E(s,w)$ and $D_E(\phi, w)$ as the right hand side of formulas  \eqref{E-L-expansion} and \eqref{Cusp-expansion}, we need to construct a vector $v_w$ in the principal series representation satisfying certain properties. We construct such a vector and make computations in a well-know model of induced representations of $\PGLR$.
\subsubsection{The plane realization} The basic affine space $Z_\8$ is  isomorphic to the punctured plane $\br^2\setminus 0$. This leads to the standard realization of the principal series representation in homogeneous functions on the plane.
For a complex parameter $\tau\in\bc$ and $\eps\in\{0,1\}$, the (smooth part of the) representation $\pi_{\tau,\eps}$ of principal series  has the realization in the space of homogeneous functions on $\br^2\setminus 0$ of the homogeneous degree $\tau-1$.  The twisted action $\pi_{\tau,\eps}(g)f(t)=f(g\inv t)|\det g|^{\frac{\tau-1}{2}}\det(g)^\eps$, $t\in  \br^2\setminus 0$,  defines a representation of   $\GL_2(\br)$ with the trivial center character, and hence defines a
representation of $\PGL_2(\br)$. The trivial representation is the subrepresentation for $\tau=1$ , $\eps=0$ (and the quotient for $\tau=-1$). The standard Casimir operator acts on $\pi_{\tau,\eps}$ by  multiplication by the scalar $\frac{1-\tau^2}{4}$. The dual representation to $\pi_{\tau,\eps}$  could be naturally identified with $\pi_{-\tau,\eps}$. Representations  $\pi_{\tau,\eps}$ are unitarizable for $\tau\in i\br\cup (-1,1) $.

Taking the restriction of functions on $\br^2\setminus 0$  to the circle $S^1\subset \br^2\setminus 0$, we obtain the circle (or compact) model for the space of $\pi_{\tau,\eps}$. This means that we realize the space of the representation as the space of smooth even functions $C_{ev}^\8(S^1)$ on the circle $S^1$ (or on $K_\8\simeq S^1$). Hence in such a model a $K_\8$-equivariant functional  is given by the integration against the exponent $e^{in\theta}$, i.e., the scalar product with a norm one $n$-th $K_\8$-type.

  Taking the restriction of  functions on the plane to a line $L\subset \br^2\setminus 0$, we obtain a line (or unipotent) model for the space of $\pi_{\tau,\eps}$.

An easy computation shows that in the above described normalization of the principal series and the identification $Z_\8\simeq \br^2\setminus 0$, the infinity component $E_{s,\8}$ is isomorphic to the representation of the principal series  with the parameter $\tau=1-2s$.

In what follows we will treat ``even" representations only  (i.e., $\eps=0$). The treatment of ``odd" representations is identical.  Hence in what follows, we denote representations of $\PGLR$ by $\pi_\tau$ suppressing $\eps$. We note that, for $\PSLZ$, representations appearing as Eisenstein series are automatically even.

\subsubsection{Test vectors}\label{test-vectors-K}  In order to realize the series $D_E(s,w)$ (and the  corresponding series $D_E(\phi,w)$), we construct the test vector $v_w$ in the representation of the principal series $\pi_{1-2s}$  (respectively in $\pi_\tau$)  of $G_\8=\PGL_2(\br)$  with the $K_\8$-types components satisfying $\hat v_w(n):= \langle v_w,e_n\rangle= |n|^{-w}$ for an even integer $n\not=0$, and $\hat v_w(0)= 1$. Clearly,  such a vector exists in an appropriate completion of the corresponding smooth representation. For the unitary principal series $\pi_\tau$, the vector $v_w$ belongs to the $L^2$-Sobolev space $S_\s(\pi_{\tau})$ of index $ \s=Re(w)-1/2$  (see \cite{BR}).
Moreover, it is easy to see that such a vector has ``local" singularities in the natural spherical model of the representation. This fact is central for our approach. We now describe the construction and the structure of the test vector.

The (smooth part of the) representation $\pi_{1-2s}$ of principal series  has the above mentioned  realization in the space  $C_{ev}^\8(S^1)$ of smooth even functions  on the circle $S^1$.
 We denote by $\theta$ the parameter on $S^1$ (e.g., $\theta\in [0,2\pi)$), and by $e_n(\theta)=e^{in\theta}$ the standard orthonormal basis. For  a smooth even function $f\in C_{ev}^\8(S^1)$, we denote  by $\hat f(n)$ its Fourier coefficients. This defines the isometry $\hat\  : L^2_{ev}(S^1)\to L^2(\bz)$. Hence for any $w\in \bc$ with $Re(w)>1/2$, we have a function $v_w\in L^2_{ev}(S^1)$ such that $\hat v_w(n)=|n|^{-w}$ for $n\not= 0$, and $\hat v_w(0)=1$.  For $w\leq 1/2$, we should view $v_w$ as a distribution on $C_{ev}^\8(S^1)$.
 It turns out that it is not convenient to work directly with the vector $v_w$ since it does not ``localize" (i.e., it is not supported in a small neighborhood of $\theta=0$ -- the fixed point of a Borel subgroup). Instead we construct a vector $u_w$ (i.e., a function in $C_{ev}^\8(S^1)$) which has small support and asymptotically has essentially the same Fourier coefficients as $v_w$.

Well-known properties of Fourier transform suggest that the vector with a local behavior $|\theta|^{w-1}$ near $\theta =0$ should  give us the desired Fourier coefficients $|n|^{-w}$, at least for $|n|\to \8$.   It is also well-known that, as an analytic family, the distribution $\frac{2^\frac{1-w}{2}}{\G(w/2)}\cdot |t|^{w-1}$ (on $\br$) and its Fourier transform $\frac{2^{w/2}}{\G\left(\frac{1-w}{2}\right)}\cdot |\xi|^{-w}$ behave better than the distribution $|t|^{w-1}$ and its Fourier transform (see \cite{G1}).  This explains our multiplication of the series $D_E(s,w)$ by the factor  $\frac{2^{w/2}}{\G\left(\frac{1-w}{2}\right)}$.

Let $f\in C^\8_{ev}(S^1)$ be a smooth even function which is supported in a small neighborhood (to be specified later) of  points $\theta=0,\ \pi$, and $f(\theta)\equiv 1$ in some (smaller) neighborhood of  $0,\ \pi$.
We consider the vector in the circle model given by
\begin{equation}
u_w(\theta)=\frac{2^\frac{1-w}{2}}{\G(w/2)}\cdot|\theta|^{w-1}f(\theta)\ ,
\end{equation} for $|\theta|$ near $0$, and then extended to an even function on $S^1$ (i.e., we define $u_w$ near $\theta=0$ and then extend it to an even function on $S^1$).  We have then, for an even integer $n\not=0$,
\begin{equation}\label{F-transf-u}\hat u_w(n)=\int_{S^1} u_w(\theta)e^{-in\theta}d\theta=\frac{2^{w/2}}{\G\left(\frac{1-w}{2}\right)}\cdot|n|^{-w}\left[1+r(f,w,n)\right]\ ,
 \end{equation}
 where $r(f,w,n)$ is a holomorphic function in $w$  for every $n$, which is decaying at least as  $|n|\inv$ for every fixed $f$ and $w$. Moreover, the function $r$ has an asymptotic expansion in $|n|\inv$ with coefficients  effectively bounded in terms of $w$ and derivatives of $f$. Namely, we have for any $N\geq 1$ and $n\not=0$,
 \begin{equation}\label{u-hat-asymp-expansion}r(f,w,n)=\sum_{k=1}^{N}c_k(f,w)|n|^{-k}+r_N(f,w,n)\ ,
 \end{equation} for some coefficients $c_k(f,w)$ holomorphically depending on $w$ for a fixed $f$. Here the remainder satisfies the bound
 \begin{equation}\label{remainder-asymp-expansion}|r_N(f,w,n)|\leq C_N(f,w)|n|^{N+1}\ ,
 \end{equation} for a constant $C_N(f,w)$ depending on $w$ and $f$.

 The relation \eqref{F-transf-u} is valid for $Re(w)>0$, but could be extended to the whole $\bc$ if we view the family of functions $u_w$ as a distribution analytically depending on $w\in\bc$ for a fixed $f$.

Together with the relation \eqref{E-L-expansion} and known properties of $K$-finite Eisenstein series (moderate growth in the type and analyticity in $s$), the relation \eqref{F-transf-u} implies that
\begin{equation}\label{E-to-D}
E_s(u_w,e)=\sum_{n} a_n(E(s))\cdot \hat u_w(n)=\tilde D_E (s,w)+R(f,s,w)\ ,
\end{equation} where $R(f,s,w)=\sum_{n\not= 0} a_n(E(s))\cdot r(f,w,n)$. Here we denote by $a_n(E(s))= L(s,\chi_{-n})$ the corresponding coefficients for the Eisenstein series.
 This relation holds as long as $E_s(u_w,e)$ is well-defined. The theory of smooth Eisenstein series (\cite{BK}, \cite{L}) implies that the value at a point for an Eisenstein series for a non-$K$-finite data  is well-defined as long as the defining vector is smooth enough (e.g., belongs to a certain Sobolev space). In particular,  $E_s(u_w,e)$ is well-defined for $Re(w)> T(s)$ with some $T(s)\in\br$ {\it depending} on $s$ (e.g., for unitary Eisenstein series $E(s)$, $Re(s)=1/2$, we can take $T(s)=1$ although this is immaterial to us). For  $Re(w)> T(s)$, the series $D_E (s,w)$ is absolutely convergent. We point out the crucial fact for us  that the series $R(f,s,w)$ is absolutely convergent in a bigger domain $Re(w)> T(s)-1$ as follows from the expansion \eqref{u-hat-asymp-expansion} and the bound \eqref{remainder-asymp-expansion}.

Hence in order to meromorphically continue the series $\tilde D_E (s,w)$ (and as a result the series $D_E (s,w)$), it is enough to analytically continue $E_s(u_w,e)$. We do this strip by strip in the variable $w$ for each {\it fixed} $s$. Analyticity in $s$ comes from the theory of smooth Eisenstein series.

For a Hecke-Maass cusp form $\phi$ in a representation $\pi_\tau$, the construction of the test vector $u_w$  is identical, and we have
\begin{equation}\label{Cusp-to-D}
\phi_{u_w}(e)=\tilde D_E (\phi,w)+R(f,\nu_\pi,w)\ ,
\end{equation} where the function $R$ is holomorphic in a bigger domain. This relation holds as long as $\phi_{u_w}(g)$ is a continuous function. According to \cite{BR} this is satisfied if the vector $u_w$ belongs to the $1/2$-Sobolev space for the representation $\pi_\tau$. The last condition holds if $Re(w)>1$. Hence \eqref{Cusp-to-D} is valid for $Re(w)>1$.

\subsubsection{Normalization of functions $F_\tau(n,z)$ and $G_\tau(\chi,z)$}\label{F-normalize} In order to construct the Dirichlet series $D_E(\phi,w)$, it seems that we have to normalize matrix coefficients $F_\tau(n,z)$, and hence the coefficients $a_n$ in the expansion  \eqref{spherical-exp} (and similarly for the expansion in \eqref{hyp-exp}). In fact it is not needed. Consider any orthonormal basis $\{e_n\}$ of $K$-types, corresponding matrix coefficients $F_\tau(n,z)=\langle \pi_\tau(g)e_0,e_n\rangle$, and the expansion \eqref{spherical-exp}
$\phi(z)=\sum_{n\in 2\bz}a_n F_\tau(n,z)$. We have then
$ D_{E}(\phi,w)=a_0+\sum\limits_{n\in 2\bz\setminus 0} a_n \langle v_w, e_n\rangle_{\pi_\tau}$. This expression does not depend on the choice of the basis $\{e_n\}$. We note that for the Eisenstein series the unfolding provides the natural choice of the basis and hence the normalization of functions $F_\tau(n,z)$. In particular, we can choose the same normalization for cusp forms as well.

The same is true for special functions $G_\tau(\chi,z)$ appearing in the hyperbolic expansion \eqref{hyp-exp}. Functions $G_\tau(\chi,z)$ could be defined via the generalized matrix coefficient $G_\tau(\bar\chi,g)=\langle\pi_\tau(g)e_0,d_\chi\rangle$, where $d_\chi$ is an explicit $\chi$-equivariant functional on the representation $V_\tau$ (e.g., see Section \ref{Real}). It is easy to write down explicitly such a functional in one of the models of the representation $\pi_\tau$. For example, in a line model such a functional is given essentially by the character $\chi$ itself, twisted by $\tau$ in order to compensate for the action of $A$ in the line model of $\pi_\tau$. Hence in such a realization the functional $d_\chi$ is given by the Mellin  transform.

\subsection{Automorphic functionals} We now switch to a more classical language of automorphic representations of $G_\br=\PGL_2(\br)$. Let  $\G=\PGL_2(\bz)$ and $X_\br=\G\sm G_\br$. We will view automorphic representations through the Frobenius reciprocity (see \cite{BR}).
\subsubsection{Cusp forms} Let $e\in X_\br$ be the class of the identity element. Evaluation at this point defines a $\G$-invariant functional on the space of smooth functions $C^\8(X_\br)$. Let $(\pi, V_\pi,\nu_\pi)$ be an automorphic cuspidal  representation, where $V_\pi$ is the space of smooth vectors of $\pi$. It is well-known that $\nu_\pi:V_\pi\to C^\8(X_\br)$. Hence we obtain the $\G$-invariant functional $\ell_\nu\in \Hom_\G(V_\pi,\bc)$ given by $\ell_\nu(v)=\nu_\pi(v)(e)$ for any $v\in V_\pi$. The Frobenius reciprocity of Gelfand and Fomin (see \cite{BR}) is the isomorphism $\Hom_G(V_\pi, C^\8(X_\br))\simeq \Hom_\G(V_\pi,\bc)$. Given $\ell\in \Hom_\G(V_\pi,\bc)$ we obtain $\nu_\ell: V_\pi\to C^\8(X_\br)$ by $\nu_\ell(v)(g)=\ell(\pi(g)v)$. It is also well-known that a cuspidal $\nu_\pi:V_\pi\to C^\8(X_\br)$ extends to the map of Hilbert spaces $\nu_\pi:L_\pi\to L^2(X_\br)$, where $L_\pi$ is the completion of $V_\pi$ with respect to invariant Hermitian norm.

\subsubsection{Eisenstein series}
Let $E_s(g)$ be the  normalized non-holomorphic Eisenstein series for $\PGL_2(\bz)$ as in Section \ref{tor-periods-set-up}. The theory of (smooth) Eisenstein series implies that the  function $E_s(g)$ generates an irreducible (for $s\not=1$) smooth representation $Eis_s\subset C^\8(X_\br)$ which is isomorphic to the (generalized) principal series representation $\pi_{1-2s}$. Hence the evaluation at $e\in X_\br$ defines a $\G$-invariant functional $\ell_{2s-1}\in \Hom_\G(\pi_{1-2s},\bc)$. The automorphic function (i.e., the automorphic realization) $\phi_v$ corresponding to a vector $v\in V_{\pi_{1-2s}}$ is  given by $\phi_v(x)=\ell_{2s-1}(\pi_{1-2s}(g)v)$.

It is natural to view the functional $\ell_{2s-1}$ as a (generalized) vector in the dual representation $\pi_{2s-1}$ (i.e., in the usual notations $\ell_{2s-1}\in V_{\pi_{2s-1}}^{-\8}$). We have the canonical pairing $\langle\cdot,\cdot\rangle :\pi_{1-2s}\otimes \pi_{2s-1}\to\bc$. We assume that this pairing  coincides with the pairing on automorphic functions. Hence we have
$\langle\ell_{2s-1}, v\rangle =E_s(v,g)|_{g=e}$ for a vector $v\in V_{\pi_{1-2s}}$.
\subsubsection{Hecke operators} We consider Hecke operators acting on automorphic representations of $G_\br$.  The theory of Hecke operators provides for each integer prime $p$, a collection of elements $\g_i\in \PGL_2(\bq)$, $0\leq i\leq p$ such that the Hecke operator acting on the space $C^\8(X_\br)$ is given by
\begin{equation}\label{Tp-formula}
T_p(f)(x)=\frac{1}{\sqrt{p}}\sum_i f(\g_i x)\ .
\end{equation}  The Eisenstein series $E_s(g)$ is an eigenvector of an operator $T_p$ with the eigenvalue $\lm_p(s)=p^{\haf-s}+p^{s-\haf}$.  For a cuspidal representation $(\pi, \nu)$, we denote the corresponding eigenvalues by $\lm_p(\pi)$ suppressing the dependance on $\nu$ (in fact, if $\pi$ stands for a representation of the adele group, then the strong multiplicity one for automorphic representations of $\PGL_2$ implies that $\nu$ determines $\pi$ uniquely).

The operator $T_p$ is a scalar operator on  the space $Eis_s$ (or in fact on any automorphic representation of $G_\br$ coming from an adele automorphic representation).
It turns out that as a result, the functional $\ell_{2s-1}$  is an eigenfunction of some operators with the same eigenvalue $\lm_p(s)$ (or rather with $\lm_p(1-s)=\lm_p(s)$) with respect to the usual action on the right by $G_\br$ on the automorphic representation $Eis_{1-s}$ (the dual of $Eis_s$). We emphasize that there is no group algebra``action" of $T_p$ on $\pi$. It is acting as a scalar operator on the automorphic realization of $\pi$! The formula \eqref{Tp-formula} does not come from the group algebra action of $\PGL_2(\br)$ on $X_\br$. However, on the special vector $\ell$, this scalar action {\it coincides} with the action of an operator (given by the same elements appearing in $T_p$) coming from the {\it group algebra action} of $\PGL_2(\br)$ on $\pi$.
\subsubsection{Hecke and Frobenius}
 Let $\nu:V_\pi\to C^\8(X)$ be an automorphic representation.  Hence  a vector $v\in V_\pi$ (in an abstract representation $\pi$) has the corresponding automorphic realization $\phi_v(x)=\nu(\pi(g)v)\in C^\8(X)$. Let $\ell_\nu\in \Hom(\pi,\bc)$ be the corresponding automorphic functional given by $\ell_\nu(v)=\nu(v)(x)|_{x=e}$ for any smooth vector $v\in V_\pi$  in $\pi$ (here $e$ is the image of the identity in $X$).  We write $\ell_\pi$ suppressing $\nu$ and view it as  a (generalized) vector in the dual representation $\pi^*$. We denote by $\langle\cdot,\cdot\rangle:\pi\otimes \pi^*\to\bc$  the natural pairing. We assume it coincides with the pairing on $X$ for  automorphic realizations of $\pi$ and $\pi^*$ (at least for cuspidal $\nu$).

 It turns out that  the functional $\ell_\pi$  is an eigenfunction of some similarly looking operators with the same eigenvalue $\lm_p(\pi)$ with respect to the {\it usual} action of $G_\br$ on the right on functions on $X_\br$.

Consider elements $\g_i\in G_\br$  from the expression \eqref{Tp-formula} for Hecke operators. Let  $\cal{T}_p= \frac{1}{\sqrt{p}}\sum_i \g_i\inv$ be an element of the group algebra of $G_\br$. We want to show that
$\pi^*(\cal{T}_p)\ell_\pi=\lm_p(\pi)\cdot \ell_\pi$. We have
\begin{eqnarray*}\langle \pi^*(\cal{T}_p)\ell_\pi,v\rangle &=&\frac{1}{\sqrt{p}}\sum_i\langle\pi^*(\g_i\inv)\ell_\pi,v\rangle =
\frac{1}{\sqrt{p}}\sum_i\langle\pi(\g_i)^*\ell_\pi,v\rangle \\
&=&\frac{1}{\sqrt{p}}\sum_i\langle\ell_\pi,\pi(\g_i)v\rangle
=\Big[\frac{1}{\sqrt{p}}\sum_i \nu(v)(x\g_i )\Big]\big|_{x=e}\\
&=&\big[T_p(\nu(v))(x)\big]\big|_{x=e}=\lm_p(\pi)\cdot\nu(v)(e)=\lm_p(\pi)\cdot\langle\ell_\pi,v\rangle
\end{eqnarray*} for any $v\in V_{\pi}$.
 Hence we have
\begin{equation}\label{Tp-act-on-l}
\pi^*(\cal{T}_p)\ell_\pi=\frac{1}{\sqrt{p}}\sum_i\pi^*(\g_i)\ell_{\pi}=\lm_p(\pi)\cdot \ell_\pi\ .
\end{equation}
We also have (essentially from the definition)
\begin{equation}\label{adjoint-action}
\langle\pi^*(\cal{T}_p)\ell_\pi,v\rangle =\langle\ell_\pi,\pi(\cal{T}'_p)v\rangle \ ,
\end{equation} where  $\pi(\cal{T}'_p)=\frac{1}{\sqrt{p}}\sum_i\pi(\g_i)$ is the action of an element in the group algebra of $G_\br$. We stress again that this is  {\it not} the action of the Hecke operator  on the automorphic representation $(\pi,\nu)$ ( since $T_p$ acts by the scalar $\lm_p(\pi)$).


\subsection{Approximate eigenvectors}\label{approx-vect-Sect}
It will be crucial for us that all elements $\g_i$ appearing in the description \eqref{Tp-formula} of Hecke operators could be chosen in the same Borel subgroup. The (convenient for us) classical choice for these elements  is $\g_i=\left(
                  \begin{array}{cc}
                    p & 0 \\
                    i & 1\\
                  \end{array}
                \right)$ for $0\leq i\leq p-1$ and $\g_p=\left(
                  \begin{array}{cc}
                    1 & 0 \\
                    0 & p\\
                  \end{array}
                \right)$.

 The main observation is  that all elements $\g_i$ preserve the point $\theta=0$ (under the natural action on $S^1$), the singularity of the vector $u_w$. As a result, the vector $u_w$ is essentially an eigenvector of certain operators related to $T_p$. We have the following elementary

 \begin{lemma*}\label{aprox-eigenval} Let $\pi_\tau$ be a principal series representation of $\PGL_2(\br)$, $p$ be an integer prime, elements $\g_i$ as above,  and $\cal{T}'_p=p^{-\haf}\sum \g_i$ be the corresponding  element in the group algebra.  For any $\s\in\bc$ and  a smooth function $g\in C^\8(S^1)$ with a small enough support around $\theta=0$, the following relation holds:
 \begin{equation}
\pi_\tau(\cal{T}'_p)(|\theta|^{\s}\cdot g(\theta))=\beta_p(\tau,\s)\cdot|\theta|^{\s} g_{p,\tau,\s}(\theta)\ ,
\end{equation} where the function $g_{\tau,\s}$ is a smooth function (in $\theta$) holomorphically  depending on $\tau$ and $\s$, and $\beta_p(\tau,\s)=p^{\tau/2-\s}+p^{\s-\tau/2}$. Moreover, we have $g_{p,\tau,\s}(0)=g(0)$ for all  $\tau$ and $\s$.
 \end{lemma*}Here we view all functions  of the variable $\theta$ (possibly depending on complex parameters $\tau$ and $\s$) as (a family of) vectors in the representation $\pi_\tau$ (realized in the same model space $C^\8_{ev}(S^1)$, but with the action of $\PGL_2(\br)$ depending on $\tau$).
\begin{proof} It is easier to write formulas in the line (or unipotent) model of $\pi_\tau$.  We recall that the principal series representation $\pi_\tau$ of $\PGL_2(\br)$ with the trivial character (i.e., representation of $\PGL_2(\br)$) has a realization in the space of functions on the real line, and the action is given by restricting the action $\pi_\tau(g)f(x)=f(g\inv x)|\det g|^{\frac{\tau-1}{2}}$, $x\in\br^2$ on the plane to the line $\{x=(t,1)\}$.  Specializing to the (lower) Borel subgroup,  we have $$\pi_\tau\left(\left(
                  \begin{array}{cc}
                    a & 0 \\
                    b & c \\
                  \end{array}
                \right)\right)f(t)=f(a\inv t/(-bt/ac+c\inv))|-bt/ac+c\inv|^{\tau-1}|ac|^{\frac{\tau-1}{2}}\ .$$

        Consider a vector $v(\theta)=|\theta|^{\s}\cdot g(\theta)$, $\theta\in S^1$ in the circle model of the representation $\pi_\tau$. Assume that $g$ is a smooth function and has small support around $\theta=0$ (and hence $|\theta|^{\s}$ makes sense). Clearly, in the line model such a vector is given by $v(t)=|t|^{\s}\tilde g_{\tau,\s}(t)$ for some smooth function
       $ \tilde g_{\tau,\s}$ supported near $t=0$, and depending holomorphically on $\tau$ and $\s$.

                Hence for $\g_i$ as above and $f\in C^\8(\br)$ supported in a small enough neighborhood of $0\in\br$,  we have $\pi_\tau(\g_i)\left(|t|^{\s}f(t)\right)= p^{-\s+\frac{\tau-1}{2}}|t|^{\s}f_{\tau,\s,i}(t)$ for $0\leq i\leq p-1$,  and $\pi_\tau(\g_p)\left(|t|^{\s}f(t)\right)= p^{\s-\frac{\tau-1}{2}}|t|^{\s}f_{ \tau,\s,p}(t)$, where functions $ f_{\tau,\s,i}$ are smooth compactly supported functions on $\br$ (depending on $\tau$, $\s$ and $\g_i$). Hence we have
$\pi_\tau(\cal{T}'_p)\left(|t|^{\s}f(t)\right)=(p^{-\s+\tau/2}+p^{\s-\tau/2})|t|^{\s}g_{p,\tau,\s}(t)$ for some smooth function $g_{p,\tau,\s}$. Taking the limit $t\to 0$ on both sides, we obtain the last claim in the lemma.
\end{proof}

\section{Meromorphic continuation}

\subsection{} We have the following main result
\begin{theorem*}\label{E-value-continuation} Let  $E(s,z)$ be the classical (normalized) Eisenstein series for $\G=\PGL_2(\bz)$ and $\ell_{2s-1}$ the corresponding automorphic functional on the  irreducible representation $\pi_{1-2s}$ of the principal series of $\PGL_2(\br)$. Let $u_{w,z}\in V_{\pi_{1-2s}}$  be a vector such that in the line model of $\pi_{1-2s}$ it is given by $u_{w,z}(t)=|t|^{w-1}F_z(t)$, where $w\in \bc$, $Re(w)\gg 1$,  and $F_z\in C^\8(\br)$, is a holomorphic family of smooth functions of compact support.  Then the function defined for $Re(w)\gg 1$, by $\ell_{2s-1}(u_{w,z})$ is a meromorphic function  in  $s$, $w$ and $z$.

The same claim holds for the function $\ell_{\pi}(u_{w,z})$ associated to a Hecke-Maass cuspidal representation $\pi$.
\end{theorem*}
In other words, the value at the identity for the Eisenstein series $E(s,u_{w,z},g)$ is a meromorphic function in  $s$, $w$ and $z$, and the same is true for the cuspidal function $\phi_{u_{w,z}}(g)$ evaluated at $g=e$.

\begin{proof} We first treat the case of cusp forms and then discuss a more delicate case of the Eisenstein series. The main difference between these two cases concerns  the issue  of boundness of the corresponding automorphic functional in an appropriate norm. For cusp forms,  there  is a clear answer in terms of Sobolev norms provided by \cite{BR}. For the Eisenstein series, we will use the Fourier expansion instead.

Let $\{F_z\}_{z\in Z}$ be an analytic family of compactly supported smooth functions on $\br$, and $u_{w,z}(t)=|t|^{w-1}F_z(t)$ the corresponding family of functions  which we view as an analytic family of vectors in the line model of an appropriate representation of the principal series. Consider an automorphic cuspidal representation $\pi\simeq \pi_\tau$ of the principal series, and the corresponding automorphic functional $\ell_\pi$. The main theorem of  \cite{BR} claims that the functional $\ell_{\pi}$ belongs to some Sobolev space completion of the dual representation $\pi^*\simeq\pi_{-\tau}$. This implies that  the value of the corresponding automorphic function at the identity, which is  given by the pairing $$\phi_{u_{w,z}}(e)=\langle \ell_{\pi}, |t|^{w-1}F_z(t)\rangle \ ,$$ is well-defined for $Re(w)\geq T$ for some $T>0$ which is large enough. Moreover, the function $\phi_{u_{w,z}}(e)$ is analytic in  parameters $w\in\bc$ and $z\in Z$ wherever it is well-defined. In fact, we can choose $T=1$ since the above quoted theorem from \cite{BR} states that $\ell_\pi$ is bounded in the $L^2$-Sobolev norm of index $\s$ for any $\s>1/2$.

Consider the operator $\cal{T}'_p$ from  Lemma \ref{aprox-eigenval}, and the function  $$g_{w}(t)=\beta_p(\tau,w-1)\cdot u_{w,z}-\pi_{\tau}(\cal{T}'_p)(u_{w,z})\ .$$
 This is an analytic family of vectors in the space $V_{\pi_{\tau}}$. Lemma \ref{aprox-eigenval} implies (via the computation of the germ at $t=0$) that  $g_{w}(t)=|t|^{w}\tilde g_{p,\tau,w,z}(t)$,  where $\tilde g_{p,\tau,w,z}$ is a smooth compactly supported function analytically depending on all parameters. Hence the function $g_{w}$ belongs to the Sobolev space on which the functional $\ell_{\pi}$ is well-defined for $w$ in a bigger region $Re(w)\geq T-1$. This implies the meromorphic continuation of $\langle \ell_{\pi}, u_{w,z}\rangle $ to a bigger strip.  Namely, it follows from \eqref{adjoint-action} and \eqref{Tp-act-on-l} that
\begin{eqnarray*}
\langle \ell_{\pi},g_{w}\rangle &= &\langle \ell_{\pi},\beta_p(\tau,w-1)\cdot u_{w,z}-\pi_{\tau}(\cal{T}'_p)(u_{w,z})\rangle \\
&=&{\beta_p(\tau,w-1)}\langle \ell_{\pi},u_{w,z}\rangle -\langle \pi_{-\tau}(\cal{T}_p)\ell_{\pi},u_{w,z}\rangle \\
&=&\big[{\beta_p(\tau,w-1)}-\lm_p(\pi) \big]\cdot\langle \ell_{\pi},u_{w,z}\rangle \ .
\end{eqnarray*}
 The left hand side is defined for $Re(w)>T-1$.  Hence we obtain the meromorphic continuation of  $\langle \ell_{\pi},u_{w,z}\rangle $ to the half-plane $Re(w)> T-1$ which is to the left of the half-plane $Re(w)\geq T$ where $\langle \ell_{\pi},u_{w,z}\rangle $ was originally defined. This continuation is defined off the zero set of the function $b_p(\tau,w-1)={\beta_p(\tau,w-1)}-\lm_p(\tau)=p^{1-w+\tau/2}+p^{-1+w-\tau/2}-\lm_p(\tau)$. However, if $w_0\in\bc$ is a zero of $b_p(\tau,w-1)$ for a given $p$, we can  change the prime $p$.  For cuspidal representations, it is well-known that not {\it all} eigenvalues of Hecke operators are  of the form $\lm_p(\pi)=p^\lm+p^{-\lm}$ for the {\it same} $\lm\in\bc$ independent of $p$, and hence the above argument shows that $\tilde D_E(\phi, w)$  is holomorphic.

 For the Eisenstein series, the treatment is in principle identical. The only issue we have to resolve is what is the appropriate norm on the representation $\pi_{2s-1}$ with respect to which the functional $\ell_{2s-1}$ is bounded. Results form \cite{BR} are not directly applicable in this case since it was required there that the representation appear discretely in $L^2(\Gamma\setminus G)$. One can however deduce from the theory of smooth Eisenstein series  (e.g., \cite{BK} and \cite{L}) that the functional $\ell_{2s-1}$ is bounded in a smooth enough Sobolev norm. A more elementary treatment is also available from the Fourier expansion of the Eisenstein series (e.g., from the fact that Fourier coefficients are at most polynomial for fixed $s$). Hence  the value $$E(s,|t|^{w-1}F_z(t),e)=\langle \ell_{2s-1}, |t|^{w-1}F_z(t)\rangle \ ,$$ is well defined for $Re(w)\geq T(s)$ with $T(s)$ depending on $s$. The rest of argument goes as in the cuspidal case. Hence we obtain the meromorphic continuation of $\tilde D_E(s,w)$. This  continuation is defined off the zero set  of the function \begin{eqnarray*}b_p(1-2s,w-1)= {\beta_p(1-2s,w-1)}-\lm_p(s)\\ =p^{w+s-1/2}+p^{-w-s+1/2}-(p^{s-1/2}+p^{1/2-s})\ .
  \end{eqnarray*} Values of $w$ which are zeros of all functions $b_p(1-2s,w-1)$ for all primes  $p$ are $w=1$ and $w=2-2s$. However, once $w_0$ is a potential pole, all values $w_0-j$, $j=0,1,2,\dots $ are potential poles due to the iterative process of the continuation strip by strip. Hence the potential polar divisor of $\tilde D_E(s,w)$ is contained in the set
 \begin{eqnarray}\label{b-zeroes}  w=1-j\ {\rm and}\ w=2-2s-j\ ,  {\rm for}\  j=0,1,2,\dots\ .
 \end{eqnarray}
Using the Fourier expansion for the Eisenstein series, one can see that there are in fact poles at $w=1,\ 2-2s$.
 \end{proof}
 \begin{remark}\label{torus-char}
One can see that the main property of the test vector $u_w$ we use is that it is essentially an eigenvector in  $V_{\pi_\tau}$ for a Borel subgroup. This is achieved by considering a vector which is essentially a small piece of a multiplicative character of a torus (in that Borel subgroup). Hence it would be more natural from the point of view of representation theory to parameterize vectors $u_w$ by that character and not by the ``artificial" parameter $w$ appearing in $D_E(s,w)$. This introduces the shift by a parameter of the representation. Namely, $u_w$  corresponds to the character $diag(a,a\inv)\mapsto |a|^{-\al}$ of the diagonal subgroup, where $\al={-2+2s+2w}$.  Accordingly, the polar set \eqref{b-zeroes} takes a more symmetric form (with respect to the natural change $s\mapsto 1-s$) $\al=-2+2s-2j$ and $\al=-2s-2j$, $j=0,1,2,\dots$.

\end{remark}
\subsection{General $\Cm$-points}\label{CM} Let $\fz\in\CH$ be a $\Cm$-point corresponding to an imaginary quadratic field $E$. There exists a non-trivial element $\g_\fz\in \PGL_2(\bq)$ which fixes $\fz$. Consider the connected compact subgroup $K_\fz\subset \PGL_2(\br)$ fixing $\fz$. We have  $\g_\fz\in K_\fz$. Let $h\in\PGL^+_2(\br)$ be an element such that  $K_\fz=h\inv\operatorname{PSO}(2,\br)h$. Consider the set $S_\fz=K_\fz\cdot (1,0)^t \subset \br^2\setminus 0$  (i.e., $S_\fz=h\inv S^1$ for the standard circle $S^1$). Note that we have a rational point $s_\fz=\g_\fz\cdot (1,0)^t \in S_\fz$ on this ellipse. Let  $B_\fz\subset\PGL_2(\bq)$ be the rational Borel subgroup  having a rational  eigenvector $s_\fz$. Now we can repeat our construction of the test vector $u_w$ from Section  \ref{test-vectors-K}. We consider the orthonormal basis $\{e^\fz_n=\pi(h\inv)e_n\}_{n\in 2\bz}$ of $K_\fz$-types. This allows us to normalize corresponding  matrix coefficients by $F^\fz_{\tau}(n,g)=\langle \pi(g) e_0^\fz, e_n^\fz\rangle_{\pi_\tau}$, and obtain the expansion
$\phi(z)=\sum_{n\in 2\bz}a^\fz_n F^\fz_\tau(n,z)$ analogous to the expansion  \eqref{spherical-exp} (this is the spherical expansion of $\phi$ centered at $\fz$). The corresponding test vector is given by $u^\fz_w=\pi(h\inv)u_w$, and as a function on $S_\fz$ has the singularity at the point $s_\fz$. The proof that the vector  $u^\fz_w$ is an approximate eigenvector of Hecke operators  given in Section \ref{approx-vect-Sect} now proceeds as before once we notice that one can choose representatives $\g_i$ for a Hecke operator in the {\it rational} Borel subgroup $B_\fz$. The rest of the proof of Theorem \ref{E-value-continuation} is identical to the case we considered.

In fact, we can apply the above argument to any compact subgroup $K$, i.e., an expansion at any point $z\in\CH$. For this we have to normalize the corresponding special functions  $F^z_{\tau}(n,g)$. This should be made with the help of the test vector $u_w$ corresponding to a rational Borel subgroup, as in Section~\ref{F-normalize}. With such a normalization of special functions and as a result the automorphic coefficients $a_n^z$, the resulting Dirichlet series could be meromorphically continued as before.

\subsection{Real quadratic periods}\label{Real} Only a slight modification is  needed in order to treat real quadratic fields, and, as is well-known, adelically one can treat $\Cm$ and real quadratic fields simultaneously.

Let $\ell\subset \CH$ be a geodesic corresponding to a real quadratic field $E$ (the equivalence class of such geodesics corresponds to an appropriate class group of $E$). There exists a non-trivial  element  $\g_\ell\in \PGL_2(\bq)$ fixing $\ell$, which is conjugate  to a diagonal matrix $h\g_\ell h\inv=diag(u,u\inv)$,  where $u$ is a unit in $E$. We will assume that $\g_\ell$ generates the corresponding group of units (i.e., $u$ is a fundamental unit), and choose two independent eigenvectors $v_1$ and $v_2$ of $\g_\ell$.  Consider a character of the diagonal group $\chi_{\al}(diag(a,b))=|a/b|^{-\al}$. To any such a character, we associate the equivariant functional $d_\al:V_{\pi_\tau}\to\bc$  given in the plane model by the kernel
$d_\al(x,y)=|x|^{\al+\tau/2-1/2}|y|^{-\al+\tau/2-1/2}$ (i.e., $\pi_{-\tau}(diag(a,b))d_\al=\chi_\al(diag(a,b))d_\al$).
The functional $d^\ell_\al=\pi_{-\tau}(h)d_\al$ is $\g_\ell$-equivariant (i.e., satisfies $\pi_{-\tau}(\g_\ell)d_\al=\chi_\al(u)d_\al$). We consider characters $\chi_i=\chi_{\al_i}$, $i\in\bz$, which are trivial on the unit group, i.e., $|u|^{\al_i}=1$ (this implies that $\al_i=i\al_1$).  We define then special functions for the hyperbolic expansion \eqref{hyp-exp} by $G_\tau(i,g)=\langle\pi_\tau(g)e_0,d_{\chi_{i}}\rangle$.

Let $B(\bq)$ be a rational Borel subgroup, and $\xi\in \br^2\setminus 0$ be an eigenvector of $B(\bq)$. Consider an {\it affine} line $L\subset \br^2\setminus 0$ generated by the eigenvector $v_1$ of $\g_\ell$ (it is transversal to $\xi$). We denote by $0_L$ the point of intersection of $L$ with the line $\br v_2$, and introduce the linear parameter $t$ on $L$ such that $t=0$ corresponds to $0_L$. Let $H_{\pi_\tau}$ be the plane realization of the principal series representation $V_{\pi_\tau}$  (i.e., $H_{\pi_\tau}$ is the space of homogeneous functions of the homogeneous degree $\tau-1$). We restrict functions in $H_{\pi_\tau}$ to the affine line $L$, and obtain the standard (twisted) linear fractional action of $G$ on the line model for $\pi_\tau$. All elements in $B(\bq)$ are fixing the point $b_L=L\cap \br\xi$, which we assume corresponds to $t=1$. Hence we can repeat our  construction by choosing Hecke operators with representatives in $B(\bq)$, and construct the test vector $u_w(t)=|t-1|^{w-1}f(t-1)$ through the coordinate $t$ as in Section \ref{test-vectors-K}. The computation of the spectral expansion of $u_w$ with respect to $d^\ell_\al$ is straightforward since on $L$ the functional $d^\ell_\al$ coincides with the Mellin transform in $t$ (and hence,  $\langle u_w,d^\ell_\al \rangle$ is given by the Beta function).

\begin{acknowledgements} It is a pleasure to thank Joseph Bernstein for endless discussions concerning automorphic functions, Dorian Goldfeld who suggested to study the series $D_E(s,w)$, Roelof Bruggeman for enlightening comments, and the referee for pointing out a mistake in an earlier draft of the paper and suggesting another proof for the Eisenstein series.
\end{acknowledgements}


\begin{thebibliography}{BBB}

\bibitem[B]{B} J. Bernstein,  Analytic continuation of generalized functions with respect to a parameter, Functional Anal. Appl. 6 (1972), 273--285.
\bibitem[BK]{BK} J. Bernstein, B. Kroetz,  Smooth Frechet globalizations of Harish-Chandra modules, preprint.
\bibitem[BR]{BR} J. Bernstein, A. Reznikov,  Sobolev norms of automorphic functionals, IMRN, 2002:40, 2155--2174 (2002).
\bibitem[BLZ]{BLZ}  R. Bruggeman, L. Lewis and D. Zagier, Function theory related to the group $\PSLR$,
From Fourier Analysis and Number Theory to Radon Transforms and Geometry, Developments in Mathematics series, Springer-Verlag, to appear.


 \bibitem[Bu]{Bu} D. Bump,  Automorphic forms and representations, Cambridge University Press, 1998.

\bibitem[G]{G} P. Garrett, Standard compact periods for Eisenstein series, notes on the homepage of P. Garrett.

\bibitem[G1]{G1} I.~Gelfand, G.~Shilov, Generalized Functions. vol. 1, Academic Press, 1964.

\bibitem[Gol]{Gol} D. Goldfeld, Automorphic forms and L-functions for the group $\GL(n,\br)$, Cambridge University Press, 2006.

\bibitem[Go]{Go} A. Good, On the Taylor coefficients of cusp forms.

\bibitem[GP]{GrPr}   B.  Gross, D. Prasad, Test vectors for linear forms. Math. Ann. 291 (1991), no. 2, 343--355.

\bibitem[H]{He} E. Hecke, \"{U}ber die Kroneckersche Grenzformel f\"{u}r reelle quadratische Körper und die Klassenzahl relative-Abelscher K\"{o}rper, Mathematische Werke, 198--207.  G\"{o}ttingen 1959.

\bibitem[II]{IchIke} A. Ichino, T. Ikeda, On the periods of automorphic forms on special orthogonal groups and the Gross-Prasad conjecture. Geom. Funct. Anal. 19 (2010), no. 5, 1378--1425.

\bibitem[JC]{JC} H. Jacquet,  N. Chen, Positivity of quadratic base change L-functions. Bull. Soc. Math. France 129 (2001), no. 1, 33--90.

\bibitem[KW]{KW} M. Kimball, D. Whitehouse, Central L-values and toric periods for GL(2). Int. Math. Res. Not. IMRN 2009, no. 1, 141--191.

\bibitem[L]{L} E. Lapid, A remark on Eisenstein series, Eisenstein series and its applications, 239-249, Progress in Math., v. 258, Birkhauser 2008.

\bibitem[LG]{Lerch}   A.  Laurincikas,  R. Garunk\v{s}tis, The Lerch zeta-function. Kluwer Academic Publishers, Dordrecht, 2002.

\bibitem[P]{P}       H. Petersson, Hans Einheitliche Begründung der Vollst\"{a}ndigkeitss\"{a}tze f\"{u}r die Poincar\'{e}schen Reihen von reeller Dimension bei beliebigen Grenzkreisgruppen von erster Art.  Abh. Math. Sem. Hansischen Univ. 14, (1941), 22--60.

\bibitem[R]{R}    M. Riesz, L'int\'{e}grale de Riemann-Liouville et le probl\`{e}me de Cauchy, Acta Math. 81 (1949), 1--223.

\bibitem[Si]{Si}  C. Siegel,   Lectures on advanced analytic number theory, Bombay, Tata Institute of Fundamental Research, 1961.

\bibitem[W]{W} J.-L.Waldspurger, Sur les valeurs de fonctions L-automorphes en leur centre de sym\'{e}trie, Comp. Math., t. 54 (1985), pp. 173--242.

\end{thebibliography}
\end{document}